\documentclass[11pt]{article}
\usepackage[english]{babel}
\usepackage{amsmath,amsthm}
\usepackage{amsfonts}
\usepackage[comma,numbers,square,sort&compress]{natbib}
\usepackage{graphicx,subfigure,amsmath,amssymb,mathrsfs,amsfonts,amstext,amsthm}
\usepackage{latexsym, amssymb}
\usepackage{graphicx}
\usepackage{subfigure}
\usepackage{pgf,tikz}
\usepackage{pstricks}
\newtheorem{thm}{Theorem}[section]

\newtheorem{lem}[thm]{Lemma}

\theoremstyle{definition}
\newtheorem{defn}[thm]{Definition}
\theoremstyle{remark}
\newtheorem{rem}[thm]{Remark}
\numberwithin{equation}{section}
\begin{document}

\title{\bfseries\textrm{Noise Leads to Quasi-Consensus of Hegselmann-Krause Opinion Dynamics*} \footnotetext{*This research was supported by the National Key Basic
Research Program of China (973 program) under grant 2014CB845301/2/3,
and by the National Natural Science Foundation of China under grants No. 91427304 and 61333001}}

\author{Wei Su,\thanks{School of Science, Beijing Jiaotong University, Beijing 100044, China, {\tt
suwei@amss.ac.cn}} \and Ge Chen,\thanks{National Center for Mathematics and Interdisciplinary Sciences \& Key Laboratory of Systems and
Control, Academy of Mathematics and Systems Science, Chinese Academy of Sciences, Beijing 100190,
China, {\tt
chenge@amss.ac.cn}}
\and Yiguang Hong,\thanks{Key Laboratory of Systems and
Control, Academy of Mathematics and Systems Science, Chinese Academy of Sciences, Beijing 100190,
China, {\tt
yghong@iss.ac.cn}}
}

%
\date{}%
\maketitle
\begin{abstract}
This paper aims at providing rigorous theoretical analysis to investigate the consensus behavior of opinion dynamics in noisy environments.
It is known that the well-known
Hegselmann-Krause (HK) opinion dynamics demonstrates various agreement or disagreement behaviors in the deterministic case.  Here we strictly show how noises provide great help to ``synchronize"
the opinions of the HK model. In fact, we prove a ``critical phenomena'' of the noisy HK dynamics, that is, the opinions merge as a
quasi-consensus in finite time in noisy environment when the noise strength is below a critical value, which implies the
fragmentation phenomenon of the HK dynamics could eventually vanish in the presence of noise. On the other hand, the opinions almost surely diverge when the noise strength
exceeds the critical value.
\end{abstract}

\section{Introduction}

In the past decades, methods of multi-agent systems have been used
to investigate the collective behavior in natural and social
systems, and several multi-agent models have been proposed. One of the central issues in the
study of multi-agent systems is the consensus (or called agreement or synchronization) of the collective
behavior and various consensus criteria have been developed mainly based on the
graph theory and energy function techniques \cite{Jad2003,
Ren2005, Hong2007, Chazelle2011, Chen2014}.

Opinion dynamics are important multi-agent systems to describe the evolution and spread of social opinions, which are deeply involved with
our daily lives, and their analysis has attracted
more and more interests in many research areas, even in our control
society \cite{Castellano2009, Nicholson1992,
frie, degro,Deffuant2000,hegselmann2002, Lorenz2007, Zhang2013, tempo,Jia2015}.
To investigate opinion dynamics, various confidence-based multi-agent models have been recently proposed \cite{Deffuant2000,hegselmann2002}. One of the famous confidence-based models, called the
Hegselmann-Krause (HK) model or Krause model \cite{krause, hegselmann2002, hk2005},
where each agent has a bounded confidence (that is, each agent only
considers the opinions from its neighbor agents whose opinion values
lie within its confidence range), effectively reveals the basic mechanisms of opinion evolution, and presents rich opinion phenomena, including agreement/consensus,
polarization, and fragmentation, which are widely
found in social networks. Since the inter-agent graphs of the
HK model are state-dependent and determined by the confidence bound,
recent studies on the HK model have been mainly based on simulations
and stochastic process analysis (see \cite{Lorenz2007, Fortuna2005,
Blondel2009} or the reference therein).

It is noticed that the classical HK model is deterministic, that is, once the initial opinion value is given, the opinions will
evolve in a deterministic way. However, the actual individuals' opinions are inevitably influenced by the
randomness during opinion transmission and evolution, due to agents'
``free-will'' or information influence from social media or private sources.  Therefore,
randomness becomes an essential factor in opinion
dynamics in reality and has been studied from various viewpoints \cite{Pineda2015, Mas2010,Pineda2011,Grauwin2012,Carro2013, Pineda2013}.   Most of these studies found an interesting phenomenon that the random noise in some situations could play a positive role in enhancing the consensus or
reducing the disagreement of opinions, which implies an important potential strategy to induce the agreement of social opinions.

Although there exist many simulations on the noisy HK models (see \cite{Pineda2013}), to the best of our knowledge, there has almost been no strict theoretical analysis to this interesting phenomenon up to now. The main obstacle of establishing the complete analysis lies in that the state-dependent topologies of the evolving opinions become even more elusive under the influence of the random noise.  At the same time,
the results concerned with random effect or disturbance to the consensus of multi-agent
systems \cite{Wang2008, Touri2009, Kar2009,Touri2014}, which mainly focused on the cases
with state-independent graph topologies or given connectivity conditions of graphs (balance, joint connectivity, etc.), are not directly applicable to the complex behaviors of the
HK dynamics with noises.

The main contribution in this paper is to establish the rigorous mathematical
analysis on how random noises influence the opinion ``consensus" for the confidence-based HK opinion model.
Stepping further than some of the previous simulation results, we strictly prove a ``critical phenomenon'' that, for
any initial opinion values and any confidence threshold, the noisy
HK model will achieve quasi-consensus (a consensus concept defined for the
noisy case) in finite time when the noise strength is not
larger than half of confidence threshold, while almost surely
diverges otherwise.  This result possesses its significance in not only revealing
the evolution mechanism of social opinions, but also the application of designing intervention strategy to induce the social opinion agreement
by injecting random noise to opinion groups.

The organization of the paper is as follows.  Section \ref{Mod_sec} introduces the noisy HK
model we investigate, while Section \ref{Cri_sec} proposes an important result of quasi-consensus for i.i.d. noise.
Then Sections \ref{Suff_sec} and \ref{Nece_sec} give a sufficient and a necessary condition for quasi-consensus, respectively, with general noise.
Section \ref{Simulations} provides numerical simulations for illustration, and finally, Section \ref{Conclusions}
concludes this paper.

\section{HK Models}\label{Mod_sec}
\renewcommand{\thesection}{\arabic{section}}


It is known that the original HK model assumes the following evolution dynamics  \cite{hegselmann2002}:
\begin{equation}\label{HKnoiseless}
  x_i(t+1)=|\mathcal{N}(i, x(t))|^{-1}\sum_{j\in \mathcal{N}(i, x(t))}x_j(t),
\end{equation}
where $i\in\mathcal{V}=\{1,2,\ldots,n\}$, $x_i(t)\in [0, 1]$ is the opinion value of agent $i$ at time $t$ and
\begin{equation}\label{neigh}
 \mathcal{N}(i, x(t))=\{1\leq j\leq n\; \big|\; |x_j(t)-x_i(t)|\leq \epsilon\}
\end{equation}
is the neighbor set of agent $i$ at $t$ with $\epsilon\in (0, 1]$ representing the confidence threshold (interaction radius) and $|S|$ denoting the cardinal number of a set $S$. Though it seems to be simple, the HK model (\ref{HKnoiseless}) captures a common evolution mechanism of many practical systems and is also able to present rich phenomena of opinion dynamics.
Sometimes, all opinions can gather in one cluster and the opinion dynamics achieves consensus or agreement, but very often, opinions are split into more than one clusters, called fragmentation phenomenon, which yields the disagreement of social opinions in another aspect. A standard opinion evolution can be found in Fig. 1, where the opinion fragmentation appears with 3 subgroups.
\begin{figure}[ht]
  \centering
  \includegraphics[width=2.5in]{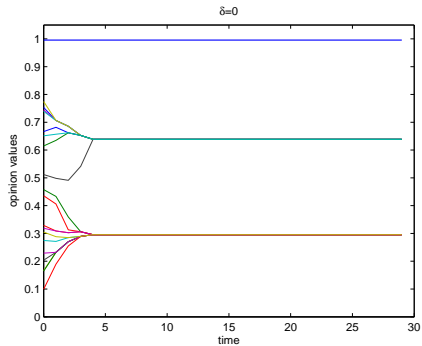}\\
  \caption{Opinion evolution of system (\ref{HKnoiseless}) with 20 agents. The initial states are randomly generated from the interval [0,1] and confidence threshold is chosen so that the opinions are divided into 3 subgroups. }\label{smallnoise0}
\end{figure}

In reality, apart from the random factor of one's free will, the individuals receive information from media and change their opinions.   In some sense, the noises may be formulated to be additive to the original opinions.
Since there are extreme opinions to bound the whole opinion dynamics, people still assume that the noisy opinion values in an opinion dynamics are limited in a bounded closed interval (as in \cite{Mas2010, Pineda2013}), say $[0, 1]$ without of loss of generality.  Therefore, the noisy HK model can be described as follows:
\begin{equation}\label{HKnoise0}
  x_i(t+1)=\left\{
           \begin{array}{ll}
             1,  & \hbox{~$x_i^*(t)>1~$} \\
             x_i^*(t),  & \hbox{~$x_i^*(t)\in[0, 1]~$} \\
             0,  & \hbox{~$x_i^*(t)<0~$}
           \end{array},~\forall i\in\mathcal{V}, t\geq 0,
         \right.
\end{equation}
where
\begin{equation}\label{xit}
  x_i^*(t)=|\mathcal{N}(i, x(t))|^{-1}\sum\limits_{j\in \mathcal{N}(i, x(t))}x_j(t)+\xi_i(t+1)
\end{equation}
with $\{\xi_i(t)\}_{i\in\mathcal{V},t> 0}$ being the random noises.

\section{Critical Noise for Quasi-Consensus}\label{Cri_sec}

To study the noisy HK model (\ref{HKnoise0})-(\ref{xit}), we slightly modify the well-known consensus concept as follows.  According to the HK model, once all the opinions locate within a confidence threshold, the agents will form a cluster and share the same average opinion in the next time.  Therefore, we define the concept of \emph{quasi-consensus} of system (\ref{HKnoise0})-(\ref{xit}) as follows.

\begin{defn}\label{robconsen}
Define
\begin{equation*}\label{opindist}
  d_{\mathcal{V}}(t)=\max\limits_{i, j\in \mathcal{V}}|x_i(t)-x_j(t)|~~\mbox{and}~~d_{\mathcal{V}}=\limsup\limits_{t\rightarrow \infty}d_{\mathcal{V}}(t).
\end{equation*}
(i) If $d_{\mathcal{V}} \leq \epsilon$, we say the system (\ref{HKnoise0})-(\ref{xit}) will reach quasi-consensus.\\
(ii) If $P\{d_{\mathcal{V}} \leq \epsilon\}=1$, we say almost surely (a.s.) the system (\ref{HKnoise0})-(\ref{xit}) will reach quasi-consensus.\\
(iii) If  $P\{d_{\mathcal{V}} \leq \epsilon\}=0$, we say a.s. the system (\ref{HKnoise0})-(\ref{xit}) cannot reach quasi-consensus.\\
(iv) Let $T=\min\{t: d_{\mathcal{V}}(t')\leq \epsilon \mbox{ for all } t'\geq t\}$.
 If $P\{T<\infty\}=1$, we say a.s. the system (\ref{HKnoise0})-(\ref{xit}) reaches quasi-consensus in finite time.
\end{defn}

For simplicity, we first give a result for the simplest i.i.d. noise. In fact, this noise condition will be relaxed in the following two sections, that is, Sections
\ref{Suff_sec} and \ref{Nece_sec}, where a necessary and a sufficient condition for quasi-consensus are provided, respectively.

\begin{thm}\label{robconsthm0}
Suppose that the noises $\{\xi_i(t)\}_{i\in\mathcal{V},t\geq 1}$ are zero-mean and non-degenerate random variables with independent and identical distribution and $E\xi_1^2(1)<\infty$.
When the initial state $x(0)\in [0,1]^n$ is arbitrarily given, we have:\\
 (i) if $P\{|\xi_1(1)|\leq \epsilon/2\}=1$, then a.s. the system (\ref{HKnoise0})-(\ref{xit}) will reach quasi-consensus in finite time; \\
 (ii) with the confidence threshold  $\epsilon\in(0,1/3]$, if
    $P\{\xi_1(1)>\epsilon/2\}>0$ and $P\{\xi_1(1)<-\epsilon/2\}>0$,
then a.s. the system (\ref{HKnoise0})-(\ref{xit}) cannot reach quasi-consensus.
\end{thm}

\begin{proof} Parts (i) and (ii) of this theorem can be deduced directly from the following Theorems \ref{Suff_thm} and \ref{Nece_thm}, respectively. 
\end{proof}

\begin{rem}
Theorem \ref{robconsthm0} shows that $\epsilon/2$ is the critical strength for the noise to induce the quasi-consensus of HK opinion dynamics: (i) states that when the noise strength is not larger than $\epsilon/2$, the noisy HK model (\ref{HKnoise0})-(\ref{xit}) will achieve quasi-consensus in finite time; and (ii) states that if the noise strength has a positive probability to be larger than $\epsilon/2$, the noisy HK model (\ref{HKnoise0})-(\ref{xit}) almost surely cannot reach quasi-consensus.
\end{rem}

\begin{rem}
The condition $\epsilon\in (0,1/3]$ in (ii) is a conservative choice in consideration of the convenience for analysis and better presenting the conclusion, and can be easily extended to $\epsilon\in(0,1)$ with suitable selection about noises.
\end{rem}

To interpret the practical significance of Theorem \ref{robconsthm0}, we consider in a community where each agent updates its opinion value not only by averaging the opinions of its neighbors, but also being slightly adjusted by information inflow from free media. This implies somehow that the free flow of information in a society could enhance the consensus of opinions. On the other hand, we can design opinion intervention strategy to reduce the disagreement by intentionally injecting random noise into the group.

\section{A Sufficient Condition for Quasi-Consensus}\label{Suff_sec}

This section will provide a sufficient condition for quasi-consensus with general independent noise, and the following is the main result.

\begin{thm}\label{Suff_thm}
Suppose $\epsilon\in(0,1]$ and the noise $\xi_i(t)$, $i\in\mathcal{V},t\geq 1$ is independent with each other and  satisfies:
 i) $P\{|\xi_i(t)|\leq \epsilon/2\}=1$;
 ii) There exist constants $a,p\in(0,1)$ such that
$P\{\xi_i(t)\geq a\}\geq p$ and $P\{\xi_i(t)\leq -a\}\geq p$.
Then for any initial state $x(0)\in [0,1]^n$, a.s. the system (\ref{HKnoise0})-(\ref{xit}) will reach quasi-consensus in finite time.
\end{thm}

To prove this result we need introduce some lemmas first.
The following result is quite straightforward, which was used in \cite{Blondel2009}.

\begin{lem}\label{monosmlem}
Suppose $\{z_i, \, i=1, 2, \ldots\}$ is a nonnegative nondecreasing (nonincreasing) sequence.  Then for any integer $n\geq 0$, the sequence
$g_n(k)=\frac{1}{k}\sum_{i=n+1}^{n+k}z_i$, $k\geq 1$, is monotonously nondecreasing (nonincreasing) with respect to $k$.
\end{lem}

\begin{rem}
Lemma \ref{monosmlem} shows that, for the nose-free HK dynamics (\ref{HKnoiseless}), the opinions of all agents keep their order from the beginning of evolution and the boundary opinions change monotonically, which can guarantee the convergence of the HK dynamics. However, these two key properties disappear in the noisy HK dynamics, and our analysis skill here is totally different with that in the noise-free case.
\end{rem}

In what follows, the ever appearing time symbol $t$ (or $T$, etc.) all refer to the random variables $t(\omega)$ (or $T(\omega)$, etc.) on the probability space $(\Omega,\mathcal{F},P)$, and will be still written as $t$ (or $T$, etc.) for simplicity.

\begin{lem}\label{robconspeci}
For protocol (\ref{HKnoise0})-(\ref{xit}) with the conditions of Theorem \ref{Suff_thm}, if a.s. there exists a finite time $T\geq 0$ such that $d_\mathcal{V}(T)\leq \epsilon$, then
$d_\mathcal{V}\leq \epsilon$ a.s..
\end{lem}

\begin{proof}
Denote $\widetilde{x}_i(t)=|\mathcal{N}(i, x(t))|^{-1}\sum_{j\in \mathcal{N}(i, x(t))}x_j(t)$, $t\geq 0$, and this denotation remains valid for the rest of the context. If $d_\mathcal{V}(T)\leq \epsilon$,
by the definition of $\mathcal{N}(i, x(t))$ we have
\begin{equation}\label{dtitera_0}
  \tilde{x}_i(T)=\frac{1}{n}\sum\limits_{j=1}^{n}x_j(T), \, \, i\in\mathcal{V}.
\end{equation}
Since $|\xi_i(t)|\leq \epsilon/2$ a.s., we obtain
\begin{equation}\label{dtitera}
\begin{split}
  d_\mathcal{V}(T+1) & =\max\limits_{1\leq i, j\leq n}|x_i(T+1)-x_j(T+1)| \\
    & \leq\max\limits_{1\leq i, j\leq n}(|\xi_i(T+1)|+|\xi_j(T+1)|) \\
    & \leq \epsilon.
\end{split}
\end{equation}
Using (\ref{dtitera_0}) and (\ref{dtitera}) repeatedly, we get $d_\mathcal{V}\leq \epsilon$.
\end{proof}

Lemma \ref{robconspeci} means that, once the opinions of all agents are located within the confidence threshold, the noises with strength no more than $\epsilon/2$ cannot separate them any more and the system (\ref{HKnoise0})-(\ref{xit}) will reach quasi-consensus.\vspace{2mm}\\
\textbf{Proof of Theorem \ref{Suff_thm}.}
Define for $t\geq 0$
\begin{equation*}
  x_{\max}(t):=\max\limits_{1\leq i\leq n}x_i(t)\,\,\,\,\,\text{and}\,\,\,\,\,x_{\min}(t):=\min\limits_{1\leq i\leq n}x_i(t).
\end{equation*}
Given any initial condition $x(0)$, if $d_\mathcal{V}(0)=x_{\max}(0)-x_{\min}(0)\leq \epsilon$, then by Lemma \ref{robconspeci}, (\ref{HKnoise0})-(\ref{xit}) is quasi-consensus. Otherwise, $d_\mathcal{V}(0)>\epsilon$. Note that there exist constants $0<a<1$ and $0<p< 1$ such that for $t\geq 1, k\in\mathcal{V}$
\begin{equation}\label{noiseposi}
  P\{\xi_k(t)\geq a\}\geq p\,\,\,\,\text{and}\,\,\,\,P\{\xi_k(t)\leq -a\}\geq p.
\end{equation}
Let $\mathcal{V}_m(t)=\{i\in\mathcal{V}|x_{\min}(t)\leq\tilde{x}_i(t)\leq x_{\min}(t)+d_\mathcal{V}(t)/2\}, \mathcal{V}_M(t)=\{j\in\mathcal{V}|x_{\min}(t)+d_\mathcal{V}(t)/2\leq\tilde{x}_j(t)\leq x_{\max}(t)\}, t\geq 0$. Then for $i\in \mathcal{V}_m(0), j\in\mathcal{V}_M(0)$, we have
\begin{equation*}
\begin{split}
  P\{x_i(1) & \geq \tilde{x}_i(0)+a\}\geq P\{\xi_i(1)\geq a\}\geq p, \\
  P\{x_j(1)  & \leq \tilde{x}_j(0)-a\}\geq P\{\xi_j(1)\leq -a\}\geq p.
\end{split}
\end{equation*}
Thus, by Lemma \ref{monosmlem} and (\ref{noiseposi}), we have
\begin{equation}\label{diffredposi}
\begin{split}
  &P\{d_\mathcal{V}(1)\leq d_\mathcal{V}(0)-2a\} \\
  \geq & P\{\xi_i(1)\geq a, \xi_j(1)\leq -a, \,\,i\in\mathcal{V}_m(0),j\in \mathcal{V}_M(0)\}\\
  \geq & p^n>0.
\end{split}
\end{equation}
Let $L=\lceil\frac{1-\epsilon}{2a}\rceil$, and according to Lemma \ref{robconspeci}, (\ref{HKnoise0})-(\ref{xit}) will reach quasi-consensus once $d_\mathcal{V}(t)\leq \epsilon$ at any time $t$. Then by repeating the above discussion, we get
\begin{equation}\label{Tconsen}
\begin{split}
&P\{d_\mathcal{V}(L+1)\leq \epsilon\}\\
\geq &P\left\{\bigcap\limits_{t=0}^{L}\{d_\mathcal{V}(t+1)\leq d_\mathcal{V}(t)-2a\}\right\}\\
\geq &P\bigg\{\bigcap\limits_{t=0}^{L}\{\xi_i(t+1)\geq a,\xi_j(t+1)\leq -a, i\in\mathcal{V}_m(t),j\in \mathcal{V}_M(t)\}\bigg\}\\
=& \prod\limits_{0\leq t\leq L} P\{\xi_i(t+1)\geq a,\xi_j(t+1)\leq -a, i\in\mathcal{V}_m(t),j\in \mathcal{V}_M(t)\}\\
\geq& p^{n(L+1)}>0.
\end{split}
\end{equation}
Define $U(L)=\{\omega:$ (\ref{HKnoise0})-(\ref{xit}) does not reach quasi-consensus in period $L\}$
and $U=\{\omega:$ (\ref{HKnoise0})-(\ref{xit}) does not reach quasi-consensus in finite time$\}$,
then by (\ref{Tconsen}),
\begin{equation*}
  P\{U(L)\}\leq 1-p^{n(L+1)}<1,
\end{equation*}
\begin{equation*}
\begin{split}
  P\{U\} & =P\bigg\{\bigcap\limits_{m=1}^{\infty}U(mL)\bigg\}=\lim\limits_{m\rightarrow \infty}P\{U(mL)\} \\
   &= \lim\limits_{m\rightarrow \infty}(1-p^{n(L+1)})^m=0,
\end{split}
\end{equation*}
and hence
$   P\{\text{ (\ref{HKnoise0})-(\ref{xit}) can reach quasi-consensus in finite time}\} 
   =  1-P\{U\}=1. \hfill\Box
$

Fig. \ref{smallnoise0} shows that the noise-free HK dynamics cannot reach consensus for some initial conditions and confidence thresholds, while Theorem \ref{Suff_thm} shows that even tiny noises can eventually make the collective opinions merge into one group. It is known that the consensus behavior of multi-agent with or without noise can be guaranteed by the graph connectivity in some sense during the evolution \cite{Jad2003,Ren2005,Wang2008,Touri2014}.   Therefore, Theorem \ref{Suff_thm} implies that the persistent noise will increase the connectivity of the multi-agent opinion dynamics even when the noise strength is very weak.

\section{A Necessary Condition for Quasi-Consensus}\label{Nece_sec}

Here we will present the analysis that when the noise strength is larger than the critical half confidence threshold, the opinions will almost surely diverge.

\begin{thm}\label{Nece_thm}
Let the initial state $x(0)\in [0,1]^n$ are arbitrarily given.  Suppose the confidence threshold  $\epsilon\in(0,1/3]$.
Assume the random noise $\{\xi_i(t), i\in\mathcal{V},t\geq 1\}$ are all zero-mean and i.i.d. with $E\xi_1^2(1)<\infty$ or independent with $\sup_{i,t}|\xi_i(t)|<\infty, a.s.$.
If there exists a lower bound $q>0$ such that
    $P(\xi_i(t)>\epsilon/2)\geq q$ and $P(\xi_i(t)<-\epsilon/2)\geq q$,
then a.s. the system (\ref{HKnoise0})-(\ref{xit}) cannot reach quasi-consensus.
\end{thm}

Before the proof of Theorem \ref{Nece_thm}, we first check how the noisy opinions behave after they get quasi-consensus. Intuitively, the evolution of opinions can fluctuate slightly when the noise is small, but the next lemma shows that, when it is long enough, the synchronized noisy opinion value will cross the whole interval [0,1] infinite times and thus, implies that the persistent noise may influence the whole opinion dynamics essentially, no matter what the nonzero noise strength is.

For an event sequence $\{B_m, \,m\geq 1\}$, let $\{B_m,\,i.o.\}$ be the set of $\bigcap\limits_{m=1}^{\infty}\bigcup\limits_{s=m}^{\infty}B_s$, where i.o. is the abbreviation of ``infinitely often''   The following lemma is important for our analysis, whose proof can be found in Appendix.

\begin{lem}\label{xvalergo}
Suppose the non-degenerate random noise $\{\xi_i(t), i\in\mathcal{V},t\geq 1\}$ are all zero-mean and i.i.d. with $E\xi_1^2(1)<\infty$ or independent with $\inf_{i,t}E\xi_i^2(t)>0$ and $\sup_{i,t}|\xi_i(t)|<\infty, a.s.$. If there exists a finite time $T$ such that $d_\mathcal{V}(t)\leq \epsilon, t\geq T$, then on $\{T<\infty\}$ it a.s. occurs
$x_i(t)= 0$, i.o. and $x_i(t)= 1$ i.o., for $i\in\mathcal{V}$.
\end{lem}

\noindent\textbf{Proof of Theorem \ref{Nece_thm}.} We only need to prove the independent case where we only need to prove that, for any $T\geq 0$, there exists $t\geq T$ a.s., such that $d_\mathcal{V}(t)>\epsilon$ when $\epsilon\leq 1/3$, i.e.,
\begin{equation*}
  P\bigg\{\bigcup\limits_{T=0}^{\infty}\{d_\mathcal{V}(t)\leq\epsilon, t\geq T\}\bigg\}=0.
\end{equation*}
Given arbitrary
$T<\infty$, we consider on $\{d_\mathcal{V}(t)\leq\epsilon, t\geq T\}$. It is easy to see that
\begin{equation}\label{deparcase1}
\begin{split}
   & P\bigg\{d_\mathcal{V}(t)\leq\epsilon,\bigcap\limits_{k=1}^\infty\bigcup\limits_{i\in\mathcal{V}}\{\xi_i(t_{m_k})>\epsilon\}\bigg\} \\
   = & P\bigg\{d_\mathcal{V}(t)\leq\epsilon,\bigcap\limits_{k=1}^\infty\bigcup\limits_{i\in\mathcal{V}}\bigcap\limits_{j\in\mathcal{V}}\{\xi_i(t_{m_k})>\epsilon,\xi_j(t_{m_k})\geq 0\}\bigg\}\\
   \leq & P\bigg\{\bigcap\limits_{k=1}^\infty\bigcap\limits_{j\in\mathcal{V}}\{\xi_j(t_{m_k})\geq 0\}\bigg\} = \prod\limits_{k=1}^\infty P\bigg\{\bigcap\limits_{j\in\mathcal{V}}\xi_j(t_{m_k})\geq 0\bigg\}\\
   \leq & \prod\limits_{k=1}^\infty\prod\limits_{j\in\mathcal{V}}(1-P\{\xi_j(t_{m_k})< -\epsilon/2\})\leq  \lim\limits_{k\rightarrow\infty}(1-q)^k=0.
\end{split}
\end{equation}
Similarly,
\begin{equation}\label{deparcase2}
  P\bigg\{d_\mathcal{V}(t)\leq\epsilon,\bigcap\limits_{k=1}^\infty\bigcup\limits_{i\in\mathcal{V}}\{\xi_i(t_{m_k})<-\epsilon\}\bigg\}=0.
\end{equation}
(\ref{deparcase1}) and (\ref{deparcase2}) imply that a.s. there are at most only finite times that the noise strength exceed $\epsilon$ when the system reach quasi-consensus. Then
by Lemma \ref{xvalergo}, there must exist an infinite time sequence $T_1,T_2,\ldots$ such that for $k=1,2,\ldots$
\begin{equation*}
  x_i\bigg(T+\sum\limits_{j=1}^kT_j\bigg)\in [\epsilon,1-\epsilon],\,\,\,i\in\mathcal{V}.
\end{equation*}
Denote $M(t)=\{u\in\mathcal{V}|x_u(t)\leq x_v(t),v\in\mathcal{V}\}, t\geq 0$ and note that
\begin{equation}\label{noiseposi2}
P\{\xi_i(t)>\epsilon/2\}\geq q\,\,\, \text{and}\,\,\,P\{\xi_i(t)<-\epsilon/2\}\geq q,
\end{equation}
for $i\in\mathcal{V}, t>0$. Thus by independence of $\{\xi_i(t),i\in\mathcal{V},t> 0\}$, we have for $\alpha\in M(t), \beta\in \mathcal{V}-M(t), j\geq 1$
\begin{equation}\label{posidepar}
\begin{split}
  & P\{d_\mathcal{V}(T+T_j+1)>\epsilon\} \\
   \geq & P\{\xi_\alpha(T+T_j+1)<-\epsilon/2, \xi_\beta(T+T_j+1)>\epsilon/2\} \\
   =& \sum\limits_{s=0}^{\infty}P\{\xi_\alpha(s+1)<-\epsilon/2,\xi_\beta(s+1)>\epsilon/2|T+T_j=s\} \cdot P\{T+T_j=s\}\\
   =& \sum\limits_{s=0}^{\infty}P\{\xi_\alpha(s+1)<-\epsilon/2,\xi_\beta(s+1)>\epsilon/2\}\cdot P\{T+T_j=s\}\\
   \geq & q^n>0.
\end{split}
\end{equation}
Denote $E_0=\Omega$ and for $j\geq 1$
\begin{equation*}
  E_j=\{\omega: d_\mathcal{V}(t)\leq\epsilon,\,t\in[T+T_j+1,T+T_j+T_{j+1}]\},
\end{equation*}
then
\begin{equation*}
\begin{split}
  P\bigg\{E_j\bigg|\bigcap\limits_{l<j}E_l\bigg\} &  \leq P\{d_\mathcal{V}(T+T_j+1)\leq\epsilon\}\leq 1-q^n<1.
\end{split}
\end{equation*}
It follows that
\begin{equation*}
\begin{split}
&P\{d_\mathcal{V}(t)\leq\epsilon,\,t\geq T\}\\
\leq &P\bigg\{\bigcap_{j\geq 1}E_j\bigg\}= \lim\limits_{m\rightarrow \infty}\prod\limits_{j=1}^mP\bigg\{E_j\bigg|\bigcap\limits_{l<j}E_l\bigg\}
\leq \lim\limits_{m\rightarrow \infty}(1-q^n)^m=0.
\end{split}
\end{equation*}
This completes the proof.
\hfill $\Box$

\section{Simulations}\label{Simulations}

In this section, we provide numerical analysis for system (\ref{HKnoise0})-(\ref{xit}) to verify our theoretical results.

Note that, when the noise strength is below the critical half of confidence threshold but very strong, the noise will play a leading role in the collective behavior. In this case, it is no surprise that strong opinion fluctuation occurs as a quasi-consensus. Therefore, in what follows, we focus on case when the noise strength is weak, which shows clearly how a ``consensus" can be achieved by infected or injected noises.
Take $n=20$, $\epsilon=0.2$ and the initial conditions uniformly distributed in $[0,1]$. The Fig. \ref{smallnoise0} shows the evolution of noise-free opinions, where three groups are formed.
Take $\xi_i(t),\,i=1,\ldots,n,\,t\geq 1$ as the noises uniformly distributed in $[-\delta,\delta]$. Consider Theorem \ref{robconsthm0} (i) and let $\delta=0.1\epsilon$, and then Fig. \ref{largernoise} shows that all opinions merge into one group.

Next, we show that, when noise strength exceeds the critical value, the opinions will diverge. For simplicity, we set the initial opinions to be identical, and show the noise with strength larger than $0.5\epsilon$ will separate them at some moment. Also, when the confidence threshold is large, the noise strength is very strong that all the opinions fluctuate sharply, though they may get separated. Therefore, we take a small confidence threshold for better illustration. Let $n=10$, $\epsilon=0.01$, all initial opinion values be $0.5$, and $\delta=0.6\epsilon>0.5\epsilon$ by Theorem \ref{robconsthm0} (ii), and then it is shown that the gathered opinions diverge at some time in Fig. \ref{largedivide}. Clearly, some of our simulation results are consistent with the simulation results given in \cite{Mas2010,Pineda2011,Pineda2013}.


\begin{figure}[ht]
  \centering
  \includegraphics[width=2.5in]{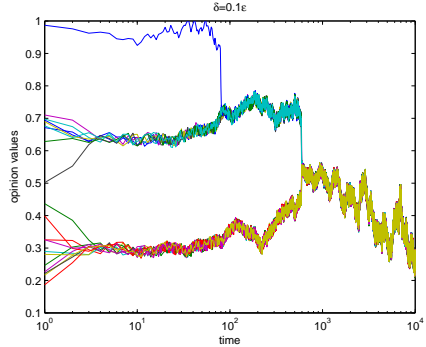}\\
  \caption{The opinion evolution of (\ref{HKnoise0})-(\ref{xit}) when all agents are noise-infected with $\delta=0.1\epsilon$. The initial states and confidence threshold are the same as those in Fig. \ref{smallnoise0}. It can be seen that the opinions merged in finite time and (\ref{HKnoise0})-(\ref{xit}) reaches quasi-consensus. }\label{largernoise}
\end{figure}

\begin{figure}[ht]
  \centering
  \includegraphics[width=2.5in]{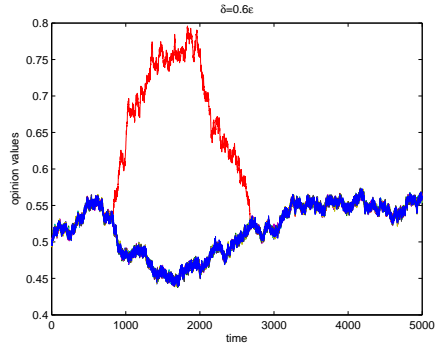}\\
  \caption{The opinion evolution of (\ref{HKnoise0})-(\ref{xit}) when all agents are noise-infected with $\delta=0.6\epsilon$. It illustrates that some opinions get separated from the others at some time when noise strength exceeds the critical value though they are gathered at the beginning.}\label{largedivide}
\end{figure}

\section{Conclusions}\label{Conclusions}

The agreement/disagreement analysis of opinion dynamics has become more and more important in recent years.   Based on the observation that random noise plays a key role in the opinion evolution in many cases, we have provided a strict analysis for the proposed noisy HK opinion model based on stochastic process and multi-agent network. Particularly, we proved that the random noise could almost surely induce the HK dynamics to achieve ``consensus'' (quasi-consensus), and also gave the critical value about the noise strength.  In fact, although the strict mathematical investigation of the noisy confidence-based opinion models is very hard and rare, we have developed a theoretical method for the noisy HK model, which may be later
extended to other confidence-based opinion dynamics.

\section*{Acknowlegment}
We would like to give our warm thanks and appreciation to Prof. Lei Guo for his valuable advice on the whole work, and Profs. Y. G. Yu and L. P. Mo for constructive discussion and modification in the proofs.

\appendix
\section{Proof of Lemma \ref{xvalergo}}
\begin{lem}\label{kollil}[\cite{Chow1997}]
(Law of the Iterated Logarithm) Let $\{X_t, t\geq 1\}$ be random variables with $EX_t=0, EX_t^2=\sigma_t^2<\infty, s_t^2=\sum\limits_{i=1}^t\sigma_i^2\rightarrow\infty$. If (i) $\{X_t,t\geq 1\}$ are i.i.d., or (ii) independent with $|X_t|\leq d_t,\,a.s.$, where the constant $d_t=o(s_t/(\log\log s_t)^{1/2})$ as $t\rightarrow\infty$, then, setting $S_t=\sum\limits_{i=1}^{t}X_i$,
\begin{equation*}
\begin{split}
  & P\bigg\{\varlimsup\limits_{t\rightarrow\infty}\frac{S_t}{s_t\sqrt{\log\log s_t}}= \sqrt{2}\bigg\}\\
  =&P\bigg\{\varliminf\limits_{t\rightarrow\infty}\frac{S_t}{s_t\sqrt{\log\log s_t}}=-\sqrt{2}\bigg\} =1.
\end{split}
\end{equation*}
\end{lem}

\noindent\textbf{Proof of Lemma \ref{xvalergo}.}
We only need to prove the case when the noise is independent, and the i.i.d. case can be obtained similarly. Without loss of generality, we assume $T=0$ a.s. and only need to prove $P\{x_i(t)= 1,  i.o.\}=1$.   Consider another modified noisy HK model as follows
\begin{equation}\label{HKnoise1}
  y_i(t+1) =  |\mathcal{N}(i, y(t))|^{-1}\sum\limits_{j\in \mathcal{N}(i, y(t))}y_j(t)+ \xi_i(t+1),
\end{equation}
for $i\in\mathcal{V}, t\geq 0$.   According to the proof of Lemma \ref{robconspeci}, we have that, for any $i\in\mathcal{V}, \, t>0$,
\begin{equation}\label{ytevolu}
\begin{split}
  & y_i(t+1) \\
  =&\frac{1}{n}\sum\limits_{j=1}^{n}y_j(t)+\xi_i(t+1)\\
  =& \frac{1}{n}\sum\limits_{j=1}^{n}(\frac{1}{n}\sum\limits_{k=1}^{n}y_k(t-1)+\xi_j(t))+\xi_i(t+1)\\
  =& \frac{1}{n}\sum\limits_{j=1}^{n}y_j(t-1)+\frac{1}{n}\sum\limits_{j=1}^{n}\xi_j(t)+\xi_i(t+1) \\
  =&\frac{1}{n}\sum\limits_{j=1}^{n}y_j(0)+\sum\limits_{k=1}^{t}\frac{1}{n}\sum\limits_{j=1}^{n}\xi_j(k)+\xi_i(t+1).
\end{split}
\end{equation}
Denote $\eta_k=\frac{1}{n}\sum\limits_{j=1}^{n}\xi_j(k),\,k\geq 1$ and suppose $\sup_{i,t}|\xi_i(t)|\leq d<\infty,a.s.,\inf_{i,t}E\xi_i^2(t)\geq c>0$.   Then $\{\eta_k, \, k\geq 1\}$ are mutually independent and
$|\eta_k| \leq d$ a.s., $E\eta_k =0$, and
\begin{equation*}
    E\eta_k^2 =\frac{1}{n^2}\sum\limits_{j=1}^{n}E\xi_j^2(k)\geq \frac{1}{n}c,
\end{equation*}
which implies
\begin{equation*}
  E\eta_k^2\in\bigg[\frac{c}{n}, \frac{d^2}{n}\bigg].
\end{equation*}
Let $S_t=\sum\limits_{k=1}^{t}\eta_k$, and then
\begin{equation*}
 s_t^2=ES_t^2=\sum\limits_{k=1}^{t}E\eta_k^2\in\bigg[\frac{tc}{n},\frac{td^2}{n}\bigg].
\end{equation*}
Hence, $\lim_{t\to\infty} s_t^2 = +\infty$ and
\begin{equation*}
  \frac{d}{(s_t/\sqrt{\log\log s_t})}  \leq \frac{d\sqrt{\log\log(\sqrt{td^2/n})}}{\sqrt{tc/n}}\rightarrow 0,
\end{equation*}
as $t\rightarrow \infty$.
According to Lemma~\ref{kollil},
\begin{equation}\label{stinfty}
   P\{\limsup_{t\rightarrow\infty} S_t=+\infty\}=  P\{\liminf_{t\rightarrow\infty} S_t=-\infty\}=1.
\end{equation}
Again with (\ref{ytevolu}), we have that, for any $y_i(0)\in[0,1]$,
\begin{equation}\label{ytinfty}
 P\{\limsup_{t\rightarrow\infty}y_i(t)=+\infty\}
   =  P\{\liminf_{t\rightarrow\infty}y_i(t)=-\infty\}=1.
\end{equation}
As a result,
\begin{equation}\label{xoscinf}
\begin{split}
  &P\{x_i(t)=1,i.o.\}\\
   =& P\bigg\{\bigcap\limits_{m=1}^{\infty}\bigcup\limits_{s=m}^{\infty}\{x_i(s)=1\}\bigg\}
    = 1-P\bigg\{\bigcup\limits_{m=1}^{\infty}\bigcap\limits_{s=m}^{\infty}\{x_i(s)<1\}\bigg\}\\
    =& 1-P\bigg\{\lim\limits_{m\rightarrow \infty}\bigcap\limits_{s=m}^{\infty}\{x_i(s)<1\}\bigg\}
    = 1-\lim\limits_{m\rightarrow \infty}P\bigg\{\bigcap\limits_{s=m}^{\infty}\{x_i(s)<1\}\bigg\}\\
    \geq & 1-\lim\limits_{m\rightarrow \infty}P\bigg\{\bigcap\limits_{s=m}^{\infty}\{y_i(s)<1|y_i(m)=x_i(m)<1\}\bigg\},
\end{split}
\end{equation}
where the second equality from last holds since the sequence $\{\bigcap\limits_{s=m}^{\infty}\{x_i(s)<1\},m\geq 1\}$ is increasing, and $P$ is a probability measure.
The last equality holds since Lemma \ref{monosmlem}, (\ref{HKnoise0}), (\ref{ytevolu}) and Lemma \ref{robconspeci} imply that $y_i(t)\leq x_i(t), t\geq m$ on $\bigcap\limits_{s=m}^{\infty}\{x_i(s)<1\}$ and hence, $\bigcap\limits_{s=m}^{\infty}\{x_i(s)<1\}\subset\bigcap\limits_{s=m}^{\infty}\{y_i(s)<1\}$, if letting $y_i(m)=x_i(m)$ for each $m\geq 1$.
From (\ref{ytinfty}), we obtain
\begin{equation}\label{brcrinequ}
\begin{split}
  P\bigg\{\bigcap\limits_{s=m}^{\infty}\{y_i(s)<1\}|y_i(m)=x_i(m)<1\bigg\}=0.
\end{split}
\end{equation}
Thus, by (\ref{xoscinf}) and (\ref{brcrinequ}), we have
$
 P\{x_i(t)= 1, \,  i.o.\}  =1. \hfill\Box
$


\begin{thebibliography}{99}

\bibitem{Jad2003}
A. Jadbabaie, J. Lin, and A. S. Morse, Coordination of groups of mobile autonomous agents using nearest neighbor rules, \emph{IEEE Trans. Autom. Control}, vol.48, no.6, pp.988-1001, June 2003.

\bibitem{Ren2005}
W. Ren, and R. W. Beard, Consensus seeking in multiagent systems under dynamically changing interaction topologies, \emph{IEEE Trans. Autom. Control}, vol.50, no.5, pp.655-661, May 2005.

\bibitem{Hong2007}
Y. Hong, L. Gao, D. Cheng, and J. Hu, Lyapunov-based approach to multiagent systems with switching jointly connected interconnection,
{\em IEEE Trans. Autom. Control}, vol.52, no.5, pp.943-948, May 2007.

\bibitem{Chen2014}
G. Chen, Z. Liu, and L. Guo, The smallest possible interaction radius for synchronization of self-propelled paricles, \emph{SIAM Review}, vol.56, no.3, pp.499-521, 2014.

\bibitem{Chazelle2011}
B. Chazelle, The total s-energy of a multiagent system. SIAM Journal on Control and Optimization, 2011, 49(4): 1680-1706.

\bibitem{Castellano2009}
C. Castellano, S. Fortunato, and V. Loreto, Statistical physics of social dynamics, \emph{Rev. Mod. Phys.}, vol.81, no.2, pp.591-646, 2009.

\bibitem{Lorenz2007}
J. Lorenz, Continuous opinion dynamics under bounded confidence: A survey, \emph{Int. J. Mod. Phys. C}, vol.18, no.12, pp.1819-1838, 2007.

\bibitem{Zhang2013}
J. Zhang and Y. Hong, Opinion evolution analysis for short-range and long-range Deffuant-Weisbuch models, \emph{Physica A: Statistical Mechanics and its Applications}, vol.392, no.21, pp.5289-5297, 2013.


\bibitem{Nicholson1992}
M. Nicholson, Rationality and the analysis of international conflict, \emph{Cambridge University Press}, 1992.

\bibitem{tempo}
P. Frasca, C. Ravazzi, R. Tempo and H. Ishii,  Gossips and prejudices: ergodic randomized dynamics in social networks,
{\it IFAC workshop on Estimation and Control of Networked Systems},
vol.4, Germany, pp 212-219, Sep. 2013.

\bibitem{Deffuant2000}
G. Deffuant, D. Neau, F. Amblard, amd G. Weisbuch, Mixing beliefs among interacting agents, \emph{Adv. Compl. Syst.}, vol.3, no.01n04, pp.87-98, 2000.

\bibitem{hegselmann2002}
R. Hegselmann and U. Krause, Opinion dynamics and bounded confidence models, analysis, and simulation, \emph{J. Artificial Societies and Social Simulation}, vol.5, no.3, pp.1-33, 2002.

\bibitem{frie}
N. Friedkin and E. Johnsen, Social influence networks and opinion change,
{\it Advances in Group Processes}, vol.16, no.1, pp.1-29, 1999.

\bibitem{Jia2015}
P. Jia, A. Mirtabatabaei, N. Friedkin, F. Bullo, Opinion Dynamics and the Evolution of Social Power in Influence Networks, \emph{SIAM Rev.}, 57(3): 367-397, 2015.

\bibitem{degro}
M. DeGroot, Reaching a consensus, \emph{Journal of American Statistical Association}, vol.69, pp.118-121, 1974.

\bibitem{hk2005}
R. Hegselmann and U. Krausem, Opinion dynamics driven by various ways of averaging. \emph{Computational Economics}, vol. 25, pp. 381-405, 2005.

\bibitem{krause}
U. Krause, A discrete nonlinear and non-automonous model of consensus formation, In S. Elaydi, G. Ldas, J. Popenda, and J. Rakowski (Eds.), \emph{Communications in Difference Equations}, Amsterdam: Gordon and Breach Publisher,  pp. 227-238, 2000.

\bibitem{Blondel2009}
V. D. Blondel, J. M. Hendrickx, and J. N. Tsitsiklis, On Krause's multi-agent consensus model with state-dependent connectivity, \emph{IEEE Trans. Autom. Control}, vol.54, no.11, pp.2586-2597, Nov. 2009.

\bibitem{Fortuna2005}
S. Fortunato, On the consensus threshold for the opinion dynamics of Krause-Hegselmann, \emph{Int. J. Mod. Phys. C}, vol.16, no.2, pp.259-270, 2005.



%
%
%
%
%
%
%




\bibitem{Pineda2015}
M. Pineda, and G. M. Buend\'{i}a, Mass media and heterogeneous bounds of confidence in continuous opinion dynamics, \emph{Physica A}, vol.420, pp. 73-84, 2015.

\bibitem{Mas2010}
M. M\"{a}s, A. Flache, and D. Helbing, Individualization as driving force of clustering phenomena in humans, \emph{PLoS Computational Biology}, vol.6, no.10, e1000959, 2010.

\bibitem{Pineda2011}
M. Pineda, R. Toral, and E. Hern¨¢ndez-Garc¨ªa, Diffusing opinions in bounded confidence processes, \emph{Eur. Phys. J. D}, vol.62, no.1, pp.109-117, 2011.

\bibitem{Grauwin2012}
S. Grauwin and P. Jensen, Opinion group formation and dynamics: Structures that last from nonlasting entities, \emph{Phys. Rev. E}, vol.85, no.6, 066113, 2012.

\bibitem{Carro2013}
A. Carro, R. Toral, and M. San Miguel, The role of noise and initial conditions in the asymptotic solution of a bounded confidence, continuous-opinion model, \emph{J. Statis. Phys.}, vol.151, no.1-2, pp.131-149, 2013.

\bibitem{Pineda2013}
M. Pineda, R. Toral, and E. Hern¨¢ndez-Garc¨ªa, The noisy Hegselmann-Krause model for opinion dynamics, \emph{Eur. Phys. J. B}, vol.86, no.12, pp.1-10, 2013.

\bibitem{Touri2014}
B. Touri, and C. Langbort. On endogenous random consensus and averaging dynamics. \emph{IEEE Trans. Contr. Network Syst.}, vol.1 no.3, pp.241-248, 2014.

\bibitem{Wang2008}
L. Wang and L. Guo, Robust consensus and soft control of multi-agent systems with noises, \emph{J. Syst. Sci. Complex.}, vol.21, no.3, pp.406-415, 2008.

\bibitem{Touri2009}
B. Touri and A. Nedic, Distributed consensus over network with noisy links, \emph{12th International Conference on Information Fusion}, Seattle, USA, pp.146-154, 2009.
%
\bibitem{Kar2009}
S. Kar and J.M.F. Moura, Distributed consensus algorithms in sensor networks with imperfect communications: Link failures and channel noise, \emph{IEEE Trans. Signal Proc.}, vol.57, no.1,
pp.355-369, 2009.
%
%

\bibitem{Chow1997}
Y. Chow and H. Teicher, \emph{Probability Theory: Independence, Interchangeability, Martingales}, Springer Science $\&$ Business Media, 2003.


%
%
%
%

%



%
\end{thebibliography}
\end{document}